\documentclass[11pt]{amsart}
\usepackage{amsmath}
\usepackage{amsfonts}
\usepackage{amssymb}
\usepackage{amsthm}
\usepackage{mathtools} 
\usepackage{latexsym}
\usepackage{bm}

\usepackage{enumerate}
\usepackage{cancel}
\usepackage{cases}
\usepackage{empheq}
\usepackage{multicol}

\usepackage{caption}
\usepackage{subcaption}
\usepackage{graphicx} 

\usepackage{color}

\usepackage[backgroundcolor=gray!30,linecolor=black]{todonotes}
   
\usepackage{mathrsfs}
\usepackage{fontenc} 
\usepackage{inputenc} 

\usepackage{verbatim}
\usepackage{float}

\theoremstyle{plain}
\newtheorem{theorem}{Theorem}[section]

\newtheorem{lemma}[theorem]{Lemma}

\theoremstyle{definition}

\theoremstyle{remark}
\newtheorem{remark}[theorem]{Remark}

\numberwithin{equation}{section} 
\numberwithin{figure}{section}   

\usepackage[square,comma,numbers,sort&compress]{natbib}
\usepackage[colorlinks=true, pdfborder={ 00 0}]{hyperref}
\hypersetup{urlcolor=blue, citecolor=red}
\usepackage{url}

\usepackage{cleveref}

\newcommand{\vect}[1]{\vec{\mathbf{#1}}}
\renewcommand{\vect}[1]{\mathbf{#1}}

\newcommand{\bfe}{\vect{e}}

\newcommand{\bu}{\vect{u}}

\newcommand{\bv}{\vect{v}}

\newcommand{\bx}{\vect{x}}
\newcommand{\by}{\vect{y}}

\newcommand{\be}{\vect{e}}

\newcommand{\field}[1]{\mathbb{#1}}

\newcommand{\nR}{\field{R}}

\newcommand{\cC}{\mathcal C}
\newcommand{\cD}{\mathcal D}

\newcommand{\cG}{\mathcal G}

\newcommand{\lp}{\left(}
\newcommand{\rp}{\right)}
\renewcommand{\hat}{\widehat}

\newcommand{\xh}{\widehat{\mathbf{x}}}
\newcommand{\vh}{\widehat{\mathbf{v}}}

\newcommand{\abs}[1]{\left\lvert#1\right\rvert}
\newcommand{\set}[1]{\left\{#1\right\}}

\newcommand{\norm}[1]{\left\|#1\right\|}

\newcounter{my_counter}
\renewcommand{\div}{\mathsf{d}}

\usepackage{placeins} 

\makeatletter
\let\latexsection\section
\renewcommand{\section}{\@ifstar {\FloatBarrier \latexsection*} {\FloatBarrier\latexsection}}

\let\latexsubsection\subsection
\renewcommand{\subsection}{\@ifstar {\FloatBarrier \latexsubsection*} {\FloatBarrier\latexsubsection}}

\newcommand{\addresseshere}{%
  \enddoc@text\let\enddoc@text\relax
}

\renewcommand{\xh}{\hat{\bx}}

\usepackage[normalem]{ulem} 

\title[Explicit $L^q$ Convergence Rates]{A Note on Explicit Convergence Rates of Nonlocal Peridynamic Operators in $L^q$-Norm}

\author{Adam Larios}
\address[Adam Larios]{Department of Mathematics, 
                University of Nebraska--Lincoln,
        Lincoln, NE 68588-0130, USA}
\email[Adam Larios]{alarios@unl.edu}

\author{Isabel Safarik}
\address[Isabel Safarik]{Department of Mathematics, 
                University of Nebraska--Lincoln,
        Lincoln, NE 68588-0130, USA}
\email[Isabel Safarik]{isabel.safarik@huskers.unl.edu}

\date{\today}



\begin{document}

\begin{abstract}
This note investigates the explicit convergence rates of nonlocal peridynamic operators to their classical (local) counterparts in $L^q$-norm. Previous results used Fourier series and hence were restricted to showing convergence in $L^2$.  Moreover, convergence rates were not explicit due to the use of the Lebesgue Dominated Convergence Theorem.  Some previous results have also used the Taylor Remainder Theorem in differential form, but this often required an assumption of bounded fifth-order derivatives.  We do not use these tools, but instead use the Hardy-Littlewood Maximal function, and combine it with the integral form of the Taylor Remainder Theorem. 
This approach allows us to establish convergence in the $L^q$-norm ($1 \leq q \leq \infty$) for nonlocal peridynamic partial derivatives, which immediately yields convergence rates for the corresponding nonlocal peridynamic divergence, gradient, and curl operators to their local counterparts as the radius (a.k.a., ``horizon'') of the nonlocal interaction $\delta \to 0$.  Moreover, we obtain an explicit rate of order $\mathcal{O}(\delta^2)$.
This result contributes to the understanding of the relationship between nonlocal and local models, which is essential for applications in multiscale modeling and simulations.
\end{abstract}

\maketitle

\section{Introduction}
\noindent In both engineering and mathematical contexts, ``peridynamic'' operators, i.e., nonlocal derivative-like operators with short-range interaction kernels, have generated significant interest since their introduction in the context of elasticity theory in the celebrated work \cite{silling2000reformulation}. Current research has expanded their use to a wide spectrum of applications, including diffusion \cite{Ignat_Rossi_2007_nlDiffusion, Oterkus_Madenci_Agwai_2014_nlDiffusion, Seleson_Gunzburger_Parks_2013_nlDiffusion}, 
image analysis \cite{Gilboa_Osher_2008_image, Gilboa_Osher_2007_nlImage},
biology \cite{Carrillo_Eftimie_Hoffmann_2015_nlBio, Painter_Bloomfield_Sherrat_Gerish_2015_nlBio}, 
nonlocal mechanics \cite{Bobaru_Duangpanya_2012_JCP, Bobaru_Duangpanya_2010_IJHMT, Hu_Ha_Bobaru_2012_nlMech, Aguiar_Patriota_2023_nlMech, Silling_Zimmermann_Abeyaratne_2003_nlMech}, 
and mathematical and computational analysis \cite{Silling_Lehoucq_2008, Chen_Gunzburger_2011_nlFEM, Jafarzadeh_Mousavi_Larios_Bobaru_2022, Jafarzadeh_Larios_Bobaru_2019, Du_Gunzburger_Lehoucq_Zhou_2012_nlDiff, Du_Ju_Tian_Zhou_2013, DElia_Gunzburger_2013, Gunzburger_Lehoucq_2010}. 
In contrast to classical local models, nonlocal peridynamic models use integral operators which not only demand less smoothness of solutions, but also capture interactions across multiple scales. 

Recently, the asymptotic comparison of nonlocal models with their local counterparts as the radius of nonlocal interaction shrinks to zero, has gained significant attention.
This nonlocal-to-local convergence of solutions has been proven in many cases, including 
the Cahn–Hilliard Equations \cite{Davoli_Scarpa_Trussardi_2021}, 
the elasticity equation \cite{Erbay_Erkip_2021}, 
the linear Navier equation \cite{Mengesha_Du_2014, Mengesha_Du_2014_bondBased}, 
diffusion equations \cite{Tian_Du_2013}, 
and other more general nonlocal problems (see, e.g., \cite{Du_Xie_Yin_2022, Foss_Radu_2019, Tian_Du_2014_asymptotic_compatible} and the references therein). In this analysis, the use of convergence of the nonlocal operators themselves to the local operators is often exploited.  Such operator-convergence results have been shown in many works, e.g., \cite{Du_Gunzburger_Lehoucq_Zhou_2013_MMMAS,DU_GUNZBURGER_LEHOUCQ_ZHOU_KUN_2013_JOE,FOSS_RADU_YU_2021,Foss_Radu_2019,DElia_Flores_Li_Radu_Yu_2020_Helmholtz}; however, these results are aimed at specialized operators, or are restricted to $L^2$ spaces, or without showing an explicit convergence rate.  
For example,  \cite{DU_GUNZBURGER_LEHOUCQ_ZHOU_KUN_2013_JOE} 
shows convergence of a peridynamic Navier operator in the $H^{-1}$ space, 
\cite{FOSS_RADU_YU_2021} shows the rate of convergence of a peridynamic divergence operator, but only in $L^2$,
\cite{Foss_Radu_2019} provides the convergence of a nonlocal deformation gradient, and
\cite{DElia_Flores_Li_Radu_Yu_2020_Helmholtz}  
shows pointwise convergence of a ``curl-of-curl''-type peridynamic operator with an implied, but not explicit, convergence rate.

Perhaps the most highly cited work in this direct is the pioneering work 
\cite{Du_Gunzburger_Lehoucq_Zhou_2013_MMMAS}, where nonlocal-to-local convergence of operators was shown by way of Fourier transform arguments and the Parseval identity, and hence was restricted to the $L^2$ case. Moreover, the convergence relied on Lebesgue Dominated Convergence Theorem, and therefore no explicit convergence rate was given. One may find an explicit convergence rate in \cite{HAAR_RADU_2022}, however this result is restricted for kernels of the form $\boldsymbol{\alpha}(x_1, \dots, x_n) = (\alpha_1(x_1), \dots, \alpha_n(x_n))$ and input functions of one dimension, and only considered the $L^\infty$ case. Hence, in this short note, we provide explicit nonlocal-to-local convergence rates for standard peridynamic operators in $L^q$ norm ($1\leq q\leq \infty$) in arbitrary dimension.  Our main result is the following.

\begin{theorem}\label{thm:div_conv}
Under the hypotheses of Lemma \ref{thm:convergence_rate},
\begin{align}\notag 
    & \norm{\cD_\omega(\bu) - \nabla\cdot \bu}_{L^q} \leq C\delta^2\norm{\bu}_{W^{3,q}} 
    \\ & \notag 
    \norm{\cG_\omega(u) - \nabla u}_{L^q} \leq C\delta^2\norm{u}_{W^{3,q}}
    \\ & \notag 
    \norm{\cC_\omega(\bu) - \nabla \times \bu}_{L^q} \leq C\delta^2\norm{\bu}_{W^{3,q}}
\end{align}
\end{theorem}
We prove this result in Section \ref{sec:rates_of_convergence}.

A common tool used in the analysis of nonlocal operators is the Taylor remainder theorem. 
While the typical approach often employs the differential form of the remainder, a central innovation in the present work lies in adopting the integral form of the remainder. This strategic choice enables us to leverage the Hardy-Littlewood Maximal function, resulting in the derivation of explicit bounds. 

In the present work, we have chosen an antisymmetric, integrable kernel function $\omega$. This is because when modeling a nonlocal gradient, many authors follow the convention used in \cite{Du_Gunzburger_Lehoucq_Zhou_2013_MMMAS}, see e.g., \cite{Jafarzadeh_Larios_Bobaru_2019,Jafarzadeh_Mousavi_Larios_Bobaru_2022,Jafarzadeh_Wang_Larios_Bobaru_2020_diffusion, Bobaru_Duangpanya_2012_JCP, Bobaru_Duangpanya_2010_IJHMT, HAAR_RADU_2022}. In certain cases, such as the study of nonlocal conservation laws, one may require a symmetric kernel, see e.g., \cite{Du_Huang_LeFloch_2017}, or even a singular kernel, see e.g., \cite{Bellido_Cueto_MoraCorral_2023_singular_kernel_nlGradients} and the references therein.

\section{Notation and Preliminaries}\label{sec:notation}
We use notation which is standard in e.g., \cite{Bobaru_Duangpanya_2012_JCP, Bobaru_Duangpanya_2010_IJHMT, Bobaru_Larios_Zhao_2021_flow}. Given two points $\bx$, $\xh\in\nR^n$ with $\bx\neq\xh$, denote
\begin{align}\notag
v:=v(\bx,t),
\quad
\bv:=\bv(\bx,t),
\quad
\hat{v}:=v(\xh,t),
\quad
\vh:=\bv(\xh,t),
\quad
\bfe:=\frac{\xh-\bx}{|\xh-\bx|}.
\end{align}
where $|\cdot|$ denotes the standard Euclidean norm; i.e., $|\bx|:=(\sum_{i=1}^n(x_i)^2)^{1/2}$ for any $\bx=(x_1,\ldots,x_n)$.
Let $\delta>0$ be a given horizon size, and define 
\begin{align}\notag
    H_{\bx}^\delta = \set{\xh\in\nR^n:|\xh-\bx|<\delta}.
\end{align}  
The volume of $H_{\bx}^\delta$ is $\alpha_n\delta^n$, where $\alpha_n$ is the volume of the unit ball in $\nR^n$.
Let $\omega: \nR^n \rightarrow[0,\infty)$ be a given non-negative, scalar valued function.  
Denote the nonlocal ($\omega$-weighted) divergence by
\begin{align}\label{nl_div}
    \cD_\omega(\bv)
    :=
    \int_{H_{\bx}^\delta}\omega(\xh-\bx) (\vh-\bv)\cdot\bfe\,d\xh,
\end{align}
the nonlocal gradient by
\begin{align}\notag 
    \cG_\omega(v)
    :=
    \int_{H_{\bx}^\delta}\omega(\xh-\bx) (\hat{v}-v)\bfe\,d\xh,
\end{align}
and the nonlocal curl:
\begin{align}\notag 
    \cC_\omega(\bv)
    := 
    \int_{H_{\bx}^\delta}\omega(\xh-\bx) (\vh-\bv)\times\bfe\,d\xh,
\end{align}

\noindent We will assume that $\omega$ is given by
\begin{align}\label{omega_form}
\omega(\bx) = 
\begin{cases}
\frac{\omega_0}{|\bx|^p}&\text{ for } 0<|\bx|<\delta,\\
0&\text{otherwise}.
\end{cases}
\end{align}
where $0<p<n$ and $\omega_0>0$ is determined by any necessary normalization criteria, e.g., \eqref{second_moment_nD}. To satisfy \eqref{second_moment_nD}, we must choose
\begin{align}\label{omega_0}
    \omega_0 = \frac{n-p+1}{\alpha_n \delta^{n-p+1}}.
\end{align}
where $\alpha_n$ is the volume of the unit ball in $\nR^n$.

In the interest of translating between different notations used in nonlocal theory, we observe that the nonlocal divergence operator we have defined \eqref{nl_div} is equivalent to the $\omega$-weighted divergence operator in \cite{Du_Gunzburger_Lehoucq_Zhou_2013_MMMAS} (see Corollary 5.2). In fact, to relate the notation in this work to the well-known notation of \cite{Du_Gunzburger_Lehoucq_Zhou_2013_MMMAS}: 
\begin{align} \notag 
    & \vect{\bm{\alpha}}_D(\bx, \by) = \be \\ \notag 
    & \phi_D(\bx) = \frac{\omega_0}{\abs{\bx}^{p+1}} \\\notag 
    & \omega_D(\bx, \by) = \omega(\by - \bx) 
\end{align}
where the subscript $D$ denotes the notation used in \cite{Du_Gunzburger_Lehoucq_Zhou_2013_MMMAS}, and the non-subscripted characters (on the right hand side) are used in the present work. 

We will abuse the notation slightly and use $\bfe$ in two contexts. When $\bfe$ appears in an integral over $H_{\bx}^\delta$, it is defined to be $\bfe = \frac{\bx - \xh}{\abs{\bx - \xh}}$. When $\bfe$ appears in an integral over $H_0^\delta$, then it is implied that $\bfe = \frac{\xh}{\abs{\xh}}$. We will use the following simplification in our proofs. Observe that, owing to the integrability of $\omega$, the divergence operator can be written in the convolutional form:
\begin{align}\label{divSimplified}
    \cD_\omega(\bu)(\bx)
    & = 
    \int_{H_{\bx}^\delta}\omega(\bx - \xh) (\bu - \hat{\bu})\cdot \bfe \, d\xh
    \\ & = \notag 
    - \int_{H_{\bx}^\delta} \hat{\bu} \omega(\bx - \xh) \cdot \bfe \, d\xh + \bu \cdot \int_{H_{\bx}^\delta}\omega(\bx - \xh) \bfe \, d\xh
    \\ & = \notag 
    - \int_{H_0^\delta} \bu(\bx - \xh) \cdot \omega(\xh)\bfe \, d\xh + \bu \cdot \int_{H_0^\delta} \omega(\xh) \bfe \, d\xh
    \\ & = \notag 
    - \int_{H_0^\delta} \bu(\bx - \xh) \cdot \omega(\xh)\bfe \, d\xh,
\end{align}
the last line due to the antisymmetry  of $\omega \bfe$.

In the forthcoming analysis, we will use the following result (see, e.g. \cite{GRAFAKOS_2014} and the references therein):

\begin{theorem}\label{HardyLittlewood}
Let locally integrable $f: \nR^n \to \nR$. Define the Hardy-Littlewood Maximal function by 
\begin{align}\label{HardyLittlewoodMaxFn}
    Mf(x) = \sup\limits_{\delta > 0}\frac{1}{\abs{H_x^\delta}} \int_{H_x^\delta} \abs{f(y)}dy
\end{align}
where $H_x^\delta$ is the $\delta$-ball centered at $x$, and $\abs{H_x^\delta}$ denotes the $n$-dimensional Lebesgue measure of $H_x^\delta$. Then, for $n \geq 1$, $1 < b \leq \infty$, and $f \in L^b(\nR^n)$, there exists a constant $C(b,n)>0$ such that 
\begin{align}\notag 
    \norm{Mf}_{L^b(\nR^n)} \leq C(b,n) \norm{f}_{L^b(\nR^n)}.
\end{align}
\end{theorem}

\section{Rates of Convergence}\label{sec:rates_of_convergence}

We prove convergence in $L^q$ for arbitrary dimension $n$.

\begin{remark}
Note that the the assumption that $u\in H^3$ gives us a convergence rate. However, to only show convergence, it is possible that we need only assume $u \in H^1(\nR)$ or $C^3(\nR)$, see for example 
\cite{Du_Gunzburger_Lehoucq_Zhou_2013_MMMAS} or \cite{HAAR_RADU_2022}.
\end{remark}

Let $\div_j u$ denote the j-th component of $\cD_\omega(\bu)$:
\begin{align}\label{div_components}
    \div_j u(\bx) = - \int_{H_0^\delta} u(\bx - \xh) \omega(\bx)e_j \, d\xh
\end{align}
where, as above, $e_j = \frac{\xh_j}{\abs{\xh}}$ for $j \in \set{1, \dots, n}$.

\begin{lemma}[Convergence of nonlocal operator components]\label{thm:convergence_rate}
Let $\bu: \nR^n \to \nR^n$ such that $\bu \in W^{3,q}$. Define $\omega$ as in \eqref{omega_form} (with $0 < p < n$) satisfying the normalization criterion
\begin{align}\label{second_moment_nD}
    \int_{H_0^\delta} \hat{x}_j \omega(\xh)e_j \, d\xh = 1. 
\end{align}
for $j = 1, \cdots, n$. Then, for $\delta > 0$ and $1 \leq q \leq \infty$,
\begin{align*}
    & \qquad \norm{\div_j u - \partial_{x_j} u}_{L^q(\nR^n)}
    \leq 
    c_0  \delta^{2} 
    \norm{u}_{W^{3,q}(\nR^n)}
\end{align*}
where $c_0=\frac{n (n - p + 1) \alpha_n}{12(n - p)}$.

\end{lemma}

\begin{proof}
To show the desired estimate, we will expand $u(\bx - \xh)$. Recall Taylor's Theorem in integral form (where we take $\xh\in H_\bx^\delta$):
\begin{align*}
    u(\bx - \xh) 
    & = 
    u(\bx) + \sum\limits_{1 \leq \abs{\alpha} \leq k} \frac{(-1)^{\abs{\alpha}}}{\alpha!} \lp \partial^\alpha u(\bx) \rp \xh^\alpha 
    \\ & \qquad + 
    \sum\limits_{\abs{\alpha} = k+1}\frac{k+1}{\alpha!}(-1)^{\abs{\alpha}}\xh^\alpha \int_0^1 (1-t)^k \lp \partial^\alpha u(\bx - t\xh) \rp \, dt.
\end{align*}
In particular, when $k = 2$:
\begin{align}\notag 
    u(\bx - \xh) 
    & = 
    u(\bx) 
    - 
    \sum\limits_{\abs{\alpha} = 1} \lp \partial^\alpha u(\bx) \rp \xh^\alpha 
    + 
    \sum\limits_{\abs{\alpha} = 2} \frac{1}{\alpha!}\lp \partial^\alpha u(\bx) \rp \xh^\alpha
    \\ & \qquad - \notag 
    \sum\limits_{\abs{\alpha} = 3} \frac{3}{\alpha!} \xh^\alpha \int_0^1 (1-t)^2 \lp \partial^\alpha u(\bx - t\xh) \rp \, dt
\end{align}
where $\bx^\alpha = x_1^{\alpha_1} x_2^{\alpha_2} \cdots x_n^{\alpha_n}$ is a scalar quantity. Hence, 
\begin{align} \label{tay_expand}
    \div_ju(\bx) 
    & = 
    \int_{H_0^\delta} u(\bx - \xh) \omega(\xh) e_j \, d\xh 
    \\ & = \notag 
    u(\bx) \int_{H_0^\delta} \omega(\bx) e_j \, d\xh 
    - 
    \sum\limits_{\abs{\alpha} = 1} \lp \partial^\alpha u(\bx) \rp \int_{H_0^\delta} \xh^\alpha \omega(\xh) e_j \, d\xh 
    \\ & \qquad + \notag 
    \sum\limits_{\abs{\alpha} = 2} \frac{1}{\alpha!}\lp \partial^\alpha u(\bx) \rp \int_{H_0^\delta} \xh^\alpha \omega(\xh) e_j \, d\xh 
    \\ & \qquad -\notag 
    \sum\limits_{\abs{\alpha} = 3} \frac{3}{\alpha!} \int_{H_0^\delta} \xh^\alpha \omega(\xh)e_j \int_0^1 (1-t)^2 \lp \partial^\alpha u(\bx - t\xh) \rp \, dt \, d\xh
\end{align}
Now, under assumption \eqref{second_moment_nD}, and using the anti-symmetry of $\omega e$, we see that
\begin{align}
    \sum\limits_{\abs{\alpha} = 1} \lp \partial^\alpha u(\bx) \rp \int_{H_0^\delta} \xh^\alpha \omega(\xh) e_j \, d\xh 
    & = \notag 
    \sum\limits_{i=1}^n \partial_{x_i} u(\bx) \int_{H_0^\delta} \hat{x}_i \omega(\xh) e_j \, d\xh
    \\ & = \notag 
    \partial_{x_j} u(\bx),
\end{align}
and when $\abs{\alpha}$ is even,
\begin{align} \notag 
    \int_{H_0^\delta} \omega(\xh) e_j \, d\xh = \int_{H_0^\delta} \xh^\alpha \omega(\xh) e_j \, d\xh = 0.
\end{align}
Applying this to \eqref{tay_expand}, we obtain
\begin{align} \notag 
    \div_j u(\bx) = -\partial_{x_j} u(\bx) - \sum\limits_{\abs{\alpha} = 3} \frac{3}{\alpha!} \int_{H_0^\delta} \xh^\alpha \omega(\xh)e_j \int_0^1 (1-t)^2 \lp \partial^\alpha u(\bx - t\xh) \rp \, dt \, d\xh
\end{align}
Adding $\partial_{x_j}u(\bx)$ to both sides:
\begin{align}\label{L2_norm_ndiv}
    & \qquad \norm{\div_ju - \partial_{x_j}u}_{L^q(\nR^n)}^q
    \\ & = \notag 
    \int_{\nR^n} \abs{ \sum\limits_{\abs{\alpha} = 3} \frac{3}{\alpha!} \int_{H_0^\delta} \xh^\alpha \omega(\xh)e_j \int_0^1 (1-t)^2 \lp \partial^\alpha u(\bx - t\xh) \rp \, dt \, d\xh }^q \, d\bx.
\end{align}
Now, to estimate \eqref{L2_norm_ndiv}, we begin by applying Fubini's Theorem and H\"older's inequality with $\frac{1}{a} + \frac{1}{b} = 1$ for some $a\geq 1$ and $b>1$ to be chosen such that $a < n/p$:
\begin{align*}
    & \qquad \int_{\nR^n} \abs{ \sum\limits_{\abs{\alpha} = 3} \frac{3}{\alpha!} \int_{H_0^\delta} \xh^\alpha \omega(\xh)e_j \int_0^1 (1-t)^2 \lp \partial^\alpha u(\bx - t\xh) \rp \, dt \, d\xh }^q \, d\bx
    \\ & = 
    \int_{\nR^n} \abs{ \sum\limits_{\abs{\alpha} = 3} \frac{3}{\alpha!} \int_0^1 (1-t)^2\int_{H_0^\delta} \xh^\alpha \omega(\xh)e_j \lp \partial^\alpha u(\bx - t\xh) \rp \, d\xh \, dt }^q \, d\bx
    \\ & \leq 
    \int_{\nR^n} \abs{ \sum\limits_{\abs{\alpha} = 3} \frac{3}{\alpha!} \int_0^1 (1-t)^2 \norm{\xh^\alpha \omega(\xh)e_j}_{L^a(H_0^\delta)}  \norm{\partial^\alpha u(\bx - t\xh)}_{L^b(H_0^\delta)} \, dt }^q \, d\bx.
\end{align*}
Since $\abs{e_j} \leq 1$ and $\xh \in H_0^\delta$
\begin{align}\label{omega_estimate}
    \norm{\xh^\alpha \omega e_j}_{L^a} \leq \delta^3\norm{\omega}_{L^a}.
\end{align}
Now, estimating $\norm{\omega}_{L^a}$ (\textit{This is equality with our choice of $\omega$}):
\begin{align}\notag 
    \norm{\omega}_{L^a}^a \leq \omega_0^a \int_{H_0^\delta} \abs{\frac{1}{\abs{\xh}^p}}^a \, d\xh = \omega_0^a \int_0^\delta \int_{\partial B(0;r)} \frac{r^{n-1}}{r^ap} \, d\sigma \, dr = \omega_0^a \frac{n \alpha_n}{n - ap}\delta^{n - ap}.
\end{align}
Applying \eqref{omega_0}, the estimate \eqref{omega_estimate} becomes
\begin{align}\label{omega_estimate_2}
    \norm{\xh^\alpha \omega e_j}_{L^a}^a \leq \frac{n\alpha_n (n-p + 1)^a}{n-ap}\delta^{3 + n-an-a}=: K_\delta 
\end{align}
Next, estimating $\norm{\partial^\alpha u(\bx - t(\cdot))}_{L^b(H_0^\delta))}$: for this estimate, in interest of using the Hardy-Littlewood Maximal Inequality \eqref{HardyLittlewood}, we use a change of variables with $z = \bx - t\xh$
\begin{align}
    \norm{\partial^\alpha u(\bx - t(\cdot))}_{L^b(H_0^\delta))}^b
    & = \notag 
    \int_{H_0^\delta} \abs{\partial^\alpha u(\bx - t\xh)}^b \, d\xh 
    \\ & = \notag 
    \alpha_n \delta_n t^{3b-n} \frac{1}{\alpha_n \delta^n} \int_{H_x^{\delta t}} \abs{\partial^\alpha u(z)}^b \, dz
    \\ & \leq \notag 
    \alpha_n \delta^n t^{3b-n} M(\abs{\partial^\alpha u(\bx)}^b)
\end{align}
Thus we compute
\begin{align*}
    & \qquad \norm{\div_j u - \partial_{x_j} u}_{L^q(\nR^n)}^q 
    \\ & \leq 
    \int_{\nR^n} \abs{ \sum\limits_{\abs{\alpha} = 3} \frac{3}{\alpha!} \int_0^1 (1-t)^2 \norm{\xh^\alpha \omega(\xh)e_j}_{L^a(H_0^\delta)}  \norm{\partial^\alpha u(\bx - t\xh)}_{L^b(H_0^\delta)} \, dt }^q \, d\bx
    \\ & \leq 
    K_\delta^q \int_{\nR^n} \abs{ \sum\limits_{\abs{\alpha} = 3} \frac{3}{\alpha!} \int_0^1 (1-t^2) \lp \alpha_n \delta^n t^{3b-n} M(\abs{\partial^\alpha u}^b) \rp^{1/b} \, dt}^q \, d\bx 
    \\ & \leq 
    K_\delta^q \alpha_n^{q/b} \delta^{qn/b} \abs{\frac{2b^2}{24b^2 - 10bn + n^2}}^q 
    \sum\limits_{\abs{\alpha} = 3} \int_{\nR^n} \lp M(\abs{\partial^\alpha u}^b ) \rp ^{q/b} \, d\bx
    \\ & \leq 
    K_\delta^q \alpha_n^{q/b} \delta^{qn/b} \abs{\frac{2b^2}{24b^2 - 10bn + n^2}}^q
    \sum\limits_{\abs{\alpha} = 3} \norm{ M(\abs{\partial^\alpha u}^b )}_{L^{q/b}}^{b/q}
    \\ & \leq 
    K_\delta^q  \alpha_n^{q/b} \delta^{qn/b} \abs{\frac{2b^2}{24b^2 - 10bn + n^2}}^q 
    \sum\limits_{\abs{\alpha} = 3} \norm{\abs{\partial^\alpha u}^b }_{L^{q/b}}^{b/q}
    \\ & \leq 
    \lp \frac{2b^2 n (n - p + 1)^a \alpha_n^{1/b + 1}}{(24b^2 - 10bn + n^2)(n - ap)} \rp^q \delta^{q\lp 3 + n + \frac{n}{b} - an - a \rp} \norm{u}_{W^{3,q}(\nR^n)}^q.
\end{align*}
Optimizing the power of $\delta^q$ (i.e., $f(a):=3 + n + \frac{n}{b} - an - a$, where $\frac{1}{a} + \frac{1}{b} = 1$) it is easy to show that $f''(a)<0$ for all $a > 0$ and the unique maximum occurs at $a = \sqrt{\frac{n}{n+1}}<1$; 
hence the maximum does not occur on the interval $1\leq a \leq \infty$, and we can thus optimize the right-hand side by sending $a\rightarrow1^+$ (and hence $b\rightarrow\infty$).  This yields
\begin{align*}
    & \qquad \norm{\div_j u - \partial_{x_j} u}_{L^q(\nR^n)}^q 
    \leq 
    \lp \frac{n (n - p + 1) \alpha_n}{12(n - p)} \rp^q \delta^{2q} 
    \norm{u}_{W^{3,q}(\nR^n)}^q.
\end{align*}
Taking the $q^{\text{th}}$ root yields the desired result in the case of $1 \leq q < \infty$. 

The $L^\infty$ case is similar, but we provide the details for the convenience of the reader. We estimate 
\begin{align*}
    & \qquad \abs{\div_j u(\bx) - \partial_{x_j} u(\bx)}
    \\ & \leq 
    \abs{ \sum\limits_{\abs{\alpha} = 3} \frac{3}{\alpha!} \int_0^1 (1-t)^2 \norm{\xh^\alpha \omega(\xh)e_j}_{L^a(H_0^\delta)}  \norm{\partial^\alpha u(\bx - t\xh)}_{L^b(H_0^\delta)} \, dt }
    \\ & \leq 
    K_\delta \abs{ \sum\limits_{\abs{\alpha} = 3} \frac{3}{\alpha!} \int_0^1 (1-t^2) \lp \alpha_n \delta^n t^{3b-n} M(\abs{\partial^\alpha u}^b) \rp^{1/b} \, dt}
    \\ & \leq 
    K_\delta \alpha_n^{1/b} \delta^{n/b} \abs{\frac{2b^2}{24b^2 - 10bn + n^2}}
    \sum\limits_{\abs{\alpha} = 3}  \lp M(\abs{\partial^\alpha u}^b ) \rp ^{1/b}
    \\ & \leq 
    K_\delta \alpha_n^{1/b} \delta^{n/b} \abs{\frac{2b^2}{24b^2 - 10bn + n^2}}
    \sum\limits_{\abs{\alpha} = 3} \norm{ M(\abs{\partial^\alpha u}^b )}_{L^\infty}^{1/b}
    \\ & \leq 
    K_\delta  \alpha_n^{1/b} \delta^{n/b} \abs{\frac{2b^2}{24b^2 - 10bn + n^2}} 
    \sum\limits_{\abs{\alpha} = 3} \norm{\partial^\alpha u }_{L^{\infty}}
    \\ & \leq 
    \lp \frac{2b^2 n (n - p + 1)^a \alpha_n^{1/b + 1}}{(24b^2 - 10bn + n^2)(n - ap)} \rp \delta^{\lp 3 + n + \frac{n}{b} - an - a \rp} \norm{u}_{W^{3,\infty}(\nR^n)}.
\end{align*}
Taking the supremum in $x$ and then optimizing the power of $\delta$ as above yields the desired result. 
\end{proof}

With this lemma in hand, Theorem \ref{thm:div_conv} immediately follows by applying the lemma to each component and summing as necessary.

\section{Acknowledgments}
The authors thank Michael Pieper, Mikil Foss, and Petronela Radu for useful conversations during the preparation of this manuscript.  
A.L. and I.S. were supported in part by NSF grants DMS-2206762 and CMMI-1953346.  A.L was also supported by USGS  grant G23AS00157 number GRANT13798170.
\begin{scriptsize}

\end{scriptsize}

\end{document}